\documentclass{amsart}
\usepackage[all,cmtip]{xy}
\usepackage{graphicx}
\usepackage{subfigure}

\newcommand{\NN}{\mathbb{N}}
\newcommand{\ZZ}{\mathbb{Z}}
\newcommand{\RR}{\mathbb{R}}
\newcommand{\CC}{\mathbb{C}}
\DeclareMathOperator{\val}{val}
\DeclareMathOperator{\LLP}{LLP}
\DeclareMathOperator{\NLLP}{NLLP}
\DeclareMathOperator{\ord}{ord}
\DeclareMathOperator{\Div}{Div}
\DeclareMathOperator{\Pic}{Pic}
\DeclareMathOperator{\ev}{ev}
\DeclareMathOperator{\SYT}{SYT}
\DeclareMathOperator{\RSYT}{RSYT}

\theoremstyle{plain}
\newtheorem{theorem}{Theorem}[section]
\newtheorem*{theorem*}{Theorem}
\newtheorem{proposition}[theorem]{Proposition}
\newtheorem{lemma}[theorem]{Lemma}

\theoremstyle{definition}
\newtheorem{definition}[theorem]{Definition}
\newtheorem{example}[theorem]{Example}
\newtheorem{notation}{Notation}
\newtheorem{algorithm}{Algorithm}
\theoremstyle{remark}
\newtheorem{remark}[theorem]{Remark}

\begin{document}
    \title[Involutions on standard Young tableaux and divisors on graphs]{Involutions on standard Young tableaux and divisors on metric graphs}
\author{Rohit Agrawal}
\email{agraw025@umn.edu}
\author{Gregg Musiker}
\email{musiker@math.umn.edu}
\author{Vladimir Sotirov}
\email{sotirov@math.wisc.edu}
\author{Fan Wei}
\email{fanwei@alum.mit.edu}
\date{\today}
\begin{abstract}
    We elaborate upon a bijection discovered by Cools, Draisma, Payne, and Robeva in
    \cite{cools2010tropical} between the set of rectangular standard Young
    tableaux and the set of equivalence classes of chip configurations on certain metric
    graphs under the relation of linear equivalence. We present an explicit
    formula for computing the $v_0$-reduced divisors (representatives of the
    equivalence classes) associated to given tableaux, and use this formula to prove
    (i) evacuation of tableaux corresponds (under the bijection) to reflecting the 
    metric graph, and (ii) conjugation of the tableaux corresponds to taking the 
    Riemann-Roch dual of the divisor.
\end{abstract}
\maketitle
\section{Introduction}
In \cite{baker2008specialization}, Baker reduces the Brill-Noether Theorem, which
concerns linear equivalence classes of divisors on a normal smooth projective curve $X$
of genus $g$, to an analogous statement regarding linear equivalence classes of
divisors on certain abstract tropical curves. In \cite{cools2010tropical}, Cools, Draisma, Payne, and Robeva then use this reduction to provide a
tropical proof of the Brill-Noether Theorem valid over algebraically closed fields of any characteristic.

The abstract tropical curves in question are compact metric graphs $\Gamma_g$, illustrated in
Figure \ref{fig:gammagintro}, consisting of a chain of $g$ concatenated loops and
designated vertices $\{v_0,\ldots,v_g\}$ such that the edge lengths $(\ell_i,m_i)$ 
satisfy a certain genericity condition.

\begin{figure}[htbp]\label{fig:gammagintro}
    \includegraphics{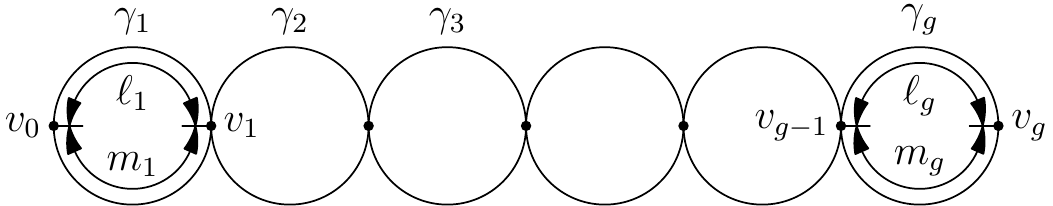}
\caption{A metric graph $\Gamma_g$ together with its designated vertices and
    edge lengths.}
\end{figure}

In the course of their proof, Cools et al.\ find a bijection from rectangular
standard Young tableaux to the linear equivalence classes of rank $r$ degree $d$ divisors on
$\Gamma_g$ that satisfy $g=(g-d+r)(r+1)$. This bijection raises the question of how
natural operations on rectangular standard Young tableaux, such as evacuation and conjugation (i.e. transpose), translate for divisors on generic metric graphs $\Gamma_g$. 
In this paper, we prove
that evacuation reflects the generic graph $\Gamma_g$ and that conjugation exchanges a divisor $c$ with its Riemann-Roch dual $K-c$.  Note that in this paper we only focus on the bijection aspect of the deep result of \cite{cools2010tropical}, and leave other applications for future work.

In Section~\ref{sec:Cools} we review both the background material necessary to understand
the bijection due to Cools et al.\, and the bijection itself. Given our combinatorial goal,
our presentation of their combinatorial results differs from the one in their paper. In particular,
we reformulate their theorem \cite[Theorem 1.4]{cools2010tropical} as a specification of an algorithm
for computing ranks of divisors on the graphs $\Gamma_g$. In this formulation, their bijection follows
as a property of that algorithm.

%

In Section~\ref{sec:res} we state and prove our first result,  Theorem~\ref{th:result}, equating evacuation and reflection, by providing formulas for the bijection and its compositions with these two operations.   In Section ~\ref{sec:res2}, we then discuss our second result, Theorem~\ref{th:result2}, equating conjugation of tableaux to the map $c \mapsto K-c$ on divisor classes. 

Before reading Sections~\ref{sec:res} and \ref{sec:res2}, the reader familiar with the combinatorics of
\cite{cools2010tropical} will need to read only Lemma~\ref{lem:park} and Subsection~\ref{subsec:alg}, in which we introduce the notation that we use to prove our main results.

\section{The bijection $\phi$ of Cools et al.}\label{sec:Cools}
In this section we will review the theory of divisors on metric graphs, and the results that
go into the derivation of the bijection of Cools et al.\ For a more detailed exposition on metric graphs,
we refer the reader to \cite{Haase2012,baker2006metrized}, which our exposition partially follows.
For an exposition of the tropical proof of the Brill-Noether theorem, we recommend the original paper
\cite{cools2010tropical}.

\subsection{Basic notions of compact metric graphs and their divisors}
In general, a \emph{metric graph} $\Gamma$ is a complete metric space such that every
point $x\in\Gamma$ is of some valence $n\in\NN$, meaning that there exists an $\epsilon$-ball
centered at $x$ isometric to the star-shaped metric
subspace $V(n,\epsilon)=\{te^{i\frac{kn}{2\pi}}\colon0\leq t<\epsilon,
	k\in\ZZ\}\subset\CC$ endowed with the path metric.
    
A \emph{model} of a compact metric graph $\Gamma$ is a finite weighted multigraph without loops $G$ such that its vertex set
$V$, which is a finite subset of $\Gamma$, satisfies the property that the connected components of $\Gamma \setminus V$ are isometric to open intervals.
The weighted edge multiset $E$ of $G$ is uniquely determined by the set of $G$'s vertices as follows.
For each connected component of $\Gamma \setminus V$ whose boundary points in $\Gamma$ are precisely $\{v,w\}\subset V$, we
add an edge between $v$ and $w$ of weight equal to the length of the isometric open interval\footnote{By abuse of notation, we will identify the edges of $G$ with the closed intervals that are isometric to the closures of
the connected components of $\Gamma \setminus V$.  By disallowing models which contain loops, we ensure that these closures are still line segments.}.
In particular, the designated vertices $\{v_0,\ldots,v_n\}$
give one model for the compact metric graph $\Gamma_g$ (illustrated in Figure~\ref{fig:gammagintro}), but that is not
the only possible model, as illustrated in Figure~\ref{fig:difmodels}.

\begin{figure}[htbp]
\centering
\subfigure{\includegraphics{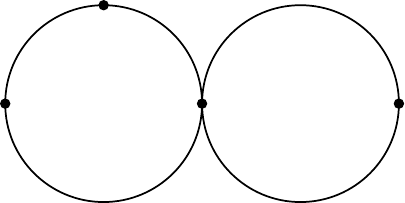}}
\hspace{1cm}
\subfigure{\includegraphics{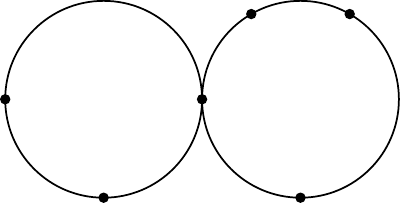}}
\caption{Two other models for the compact metric graph $\Gamma_2$.}
\label{fig:difmodels}
\end{figure}

The \emph{genus} $g$ of a compact metric graph $\Gamma$ can be defined to be $|E|-|V|+k$ where $(V,E)=G$ is a model of $\Gamma$
and $k$ is the number of connected components. It is easy to show that given two models $G_1=(V_1,E_1)$ and $G_2=(V_2,E_2)$,
both give the same genus as the model $G_3=(V_1\cup V_2,E_3)$. In particular, the genus of the graphs $\Gamma_g$ 
is precisely $g$.

A \emph{divisor} on a compact metric graph $\Gamma$ is an element
of the free abelian group $\Div(\Gamma)$ generated by the points of $\Gamma$.
There is nothing deep about the group of divisors; the deep analogy between the theory of
divisors on compact metric graphs and the theory of divisors on Riemann surfaces comes from the
definitions of rational functions on $\Gamma$, their
orders at points on $\Gamma$, and the consequent notion of equivalence of divisors,
which is strong enough for an analogue of the Riemann-Roch theorem to hold (the
so-called Tropical Riemann-Roch Theorem).

A \emph{rational function} $f$ on a compact metric graph $\Gamma$ is a
continuous function $f\colon\Gamma\to\mathbb R$ that is piecewise linear with
integer slopes in the following sense: there exists a model $G_f$ of $\Gamma$ such that the restriction
of $f$ to each edge is a linear function with integer slope. We set the \emph{order of $f$ at $x$},
$\ord_x(f)$, to be $0$ if $x$ is not a vertex of $G_f$, and otherwise we set $\ord_x(f)$ to be
the sum of the outgoing slopes along edges coming out of $x$. Note that $\ord_x(f)$ is non-zero for only finitely many
points $x$ and that $\ord_x\colon\Div(\Gamma)\to\ZZ$ is a homomorphism as $\ord_x(f+g) = \ord_x(f)+\ord_x(g)$.

Two divisors $c$ and $c'$ are said to be \emph{equivalent} if 
$c-c'=\sum_{x\in\Gamma}\ord_x(f)(x)$ for some rational function $f$. The divisors of the form
$\sum_{x\in\Gamma}\ord_x(f)$ form an abelian group called the group of \emph{principal divisors}.
The quotient of $\Div(\Gamma)$ by the group of principal divisors is denoted by $\Pic(\Gamma)$ and
consists of the equivalence classes of divisors under the above equivalence relation.

Given a divisor $c$, the \emph{degree of $c$} is defined to be
$\deg(c)=\sum_{x\in\Gamma}c(x)$. The degree is invariant under
equivalence, as any principal divisor has degree $0$. Hence every element of $\Pic(\Gamma)$
has a well-defined degree as well. Given an integer $d$, we define $\Pic^{\leq d}(\Gamma)$ to be the subset of $\Pic(\Gamma)$ with degree at most $d$.

A divisor $e=\sum_{i=1}^na_i(x_i)$ is said to be \emph{effective} if $a_i\geq0$ for all $1\leq i\leq n$. 
A divisor $c$ that is not equivalent to an effective divisor is said to have $\emph{rank $r(c)=-1$}$. A divisor
$c$ that is equivalent to an effective divisor is said to have $\emph{rank $r(c)=r$}$ if $r$ is the largest
number such that for every effective divisor $e$ of degree $r$, the divisor $c-e$ is equivalent to an effective divisor. Note that in
particular $r(c)\leq\min\{-1,\deg(c)\}$ since if $c-e$ has negative degree, then it cannot be equivalent
to an effective divisor.

The interest in this notion of rank stems from the fact that it
is invariant under equivalence of divisors and, more importantly, that it
satisfies an analogue of the Riemann-Roch theorem, first proven for finite
graphs by Baker and Norine in \cite{baker2007riemann} and subsequently generalized for
metric graphs and tropical curves independently by Gathmann and Kerber in \cite{gathmann2008riemann}
and Mikhalkin and Zharkov in \cite{mikhalkin465tropical}.

\begin{theorem}[Tropical Riemann-Roch] \label{TRR}
    Suppose that $c\in\Div(\Gamma)$ is a divisor on a compact metric
	graph. Define the canonical divisor $K$ on $\Gamma$ by
	$\displaystyle {K=\sum_{x\in\Gamma}(\val(x)-2)(x)}$. Then we have:
	\[
		r(c)-r(K-c)=\deg(c)+1-g
	\]
	where $g$ is the genus of $\Gamma$.
\end{theorem}

\subsection{The graphs $\Gamma_g$, and their $v_i$-reduced divisors}
In their paper, Cools et al.\ reduce the verification of the Brill-Noether theorem to an
analysis of the ranks of divisors on members of the following family of genus $g$ graphs.
\begin{definition}\label{def:gamma}
    The compact metric graphs $\Gamma_g$ consist of $g$ circles $\{\gamma_i\}_{1\leq i\leq g}$
    of circumferences $\{\ell_i+m_i\}$ concatenated together in such a way that there exists a model
    with vertices $\{v_i\}_{0\leq i\leq g}$ so that for every $1\leq i\leq g$, $v_{i-1}$
    and $v_i$ are designated vertices of $\gamma_i$ and the two edges of
    $\gamma_i$ joining $v_i$ and $v_{i-1}$ have lengths $\ell_i$ and $ m_i$. The
	metric on $\Gamma_g$ is the path-length metric. See Figure \ref{fig:gammagintro} for an illustration.
\end{definition}

In particular, Cools et al.\ are interested in computing the rank of an arbitrary divisor on $\Gamma_g$.
In the case where the divisor has degree greater than $2g-2$, there is a simple answer using the tropical
Riemann-Roch theorem, agreeing with the answer in the case of algebraic curves.
\begin{proposition}
	Any divisor $c$ on the compact metric graph $\Gamma_g$ of degree greater than $2g-2$ has
	rank $\deg(c)-g$.
%
\end{proposition}

In the case where the divisor has degree at most $2g-2$, the computation of the rank becomes extremely
difficult in general. For graphs $\Gamma_g$ satisfying a certain genericity condition, Cools et al.\ 
give an elementary algorithm for performing the computation. In this subsection we describe
first the input to the algorithm, which consists of certain divisors known
as $v$-reduced divisors, and second the genericity condition that the graphs $\Gamma_g$ must satisfy for the
algorithm to work correctly.

The notion of $v$-reduced divisors for graphs was first introduced by Baker and Norine in \cite{baker2007riemann}
as a slight variant of the notion of $G$-parking functions introduced by Postnikov and Shapiro in \cite{postnikov2003}.
Their importance for the theory of divisors stems from the fact that they provide a system of representatives of
the group $\Pic(\Gamma)$ of equivalence classes of divisors. Unfortunately, the language of divisors is somewhat 
clunky for describing the $v$-reduced divisors and their properties, so instead we consider a divisor 
$\sum_{i=1}^n a_i(x_i)\in\Div(\Gamma)$ to be the \emph{chip configuration} that assigns to each point $x_i$ 
the respective amount of $a_i$ chips. From here onward, we will use the terms ``chip configuration'' and ``divisor'' 
interchangeably. In particular, by the degree and rank of a chip configuration we will mean the degree or rank of that divisor.

The following definition of $v$-reduced divisors for metric graphs is due to Luo \cite{luo2011}.
\begin{definition}
	Fix a point $v$ on a connected compact metric graph $\Gamma$. A chip
	configuration $c\in\Div(\Gamma)$ is called a \emph{$v$-reduced divisor} if:
	\begin{enumerate}
		\item $c$ has a non-negative number of chips on every point except for possibly $v$;
		\item for any closed connected subset $X\subset\Gamma$ not containing $v$, 
            there exists a point $x$ in the boundary of $X$ such that the number of chips on $x$ is
            strictly smaller than the number of edges from $x$ to $\Gamma\setminus X$
            (the number of edges joining $x$ to a point
            in $\Gamma\setminus X$ for any model $G$ of $\Gamma$ such that $x$ is a vertex in $G$ and $G$ restricts to a model of $X$).
	\end{enumerate}
\end{definition}

In \cite{luo2011}, Luo generalizes the so-called burning algorithm due to Dhar \cite{dhar1990self}
for finding $v$-reduced divisors from the case of finite graphs to the case of metric graphs, and uses it to
prove that for any $v$ in a connected compact metric graph $\Gamma$, the set of $v$-reduced divisors
is a system of representatives for $\Pic(\Gamma)$.
\begin{theorem} [Theorem 2.3 of \cite{luo2011}]
	Suppose that $\Gamma$ is a connected compact metric graph, and that $v$ is a
	designated point on $\Gamma$. Then every chip configuration
	$c\in\Div(\Gamma)$ is equivalent to exactly one $v$-reduced divisor.
\end{theorem}

Using Luo's generalization of Dhar's burning algorithm, which we will not describe, one can easily
compute for any chip configuration $c$ on $\Gamma_g$ the $v_i$-reduced divisors equivalent to $c$. Since
the rank of a chip configuration is invariant under equivalence, being able to determine the rank of 
$v_i$-reduced divisors is enough to determine the rank of any chip configuration.

Cools et al.\ noticed that the $v_i$-reduced divisors on the graph $\Gamma_g$ have an elementary description
which follows almost immediately from the definitions. This description in turn suggests a compact
notation for the elements of $\Pic(\Gamma_g)$ once we identify them with $v_0$-reduced divisors. 
\begin{proposition} [Example 2.6 of \cite{cools2010tropical}]
	The $v_i$-reduced divisors on $\Gamma_g$ are precisely those chip configurations $c\in\Div(\Gamma_g)$
    for which:
	\begin{enumerate}
		\item every point different from $v_i$ has a non-negative number of
			chips;
		\item each of the cut loops to the right of $v_i$ (given by
			$\gamma_j\setminus\{v_{j-1}\}$ for $j > i$) and to the left of
			$v_i$ (given by $\gamma_j\setminus\{v_j\}$ for $j \leq i$) contain at
			most $1$ chip.
	\end{enumerate}
\end{proposition}

\begin{example}
\label{ex:running}
Consider a compact metric graph $\Gamma_6$ with edge lengths $m_1=m_2=\dots=m_6=1$ and
$\ell_1=\ell_2=\dots=\ell_6=10$ (see Figure \ref{fig:gammagintro}). The following, which we will use
as a running example, illustrates a $v_0$-reduced divisor
which has $2$ chips on $v_0$ and at most $1$ chip on every loop $\gamma_i\setminus\{v_{i-1}\}$ for $1\leq i \leq g$. Note that the extra vertices shown in the first two loops are unnecessary, but help clarify the distances in this model.

    \begin{center}
    \includegraphics{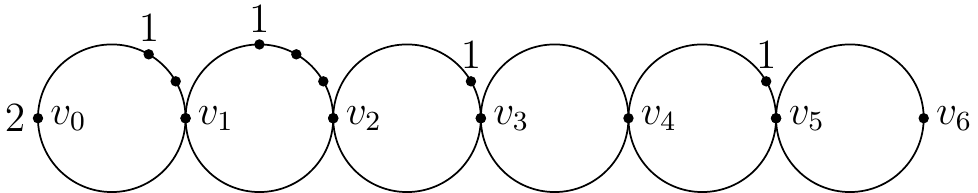}
    \end{center}
    
\end{example}

\begin{definition}\label{def:seq}
	To every $v_0$-reduced divisor on $\Gamma_g$, and hence to
	every element of $\Pic(\Gamma_g)$, we associate sequences
	$(d_0;x_1,x_2,\ldots,x_g)$ where $d_0$ is the number of chips on $v_0$, and
	$x_i\in\RR$ a distance (modulo the circumference of the
	$i^\text{th}$ loop) from $v_{i-1}$ in the counter-clockwise direction
    of the single chip on $\gamma_i\setminus\{v_{i-1}\}$, with $x_i=0$ if there is no chip (regardless of the possible existence of chips on $v_{i-1}$). These
	sequences are unique up to the modular equivalences $x_i\bmod(\ell_i+m_i)$.
	
    For Example \ref{ex:running} this sequence is $(2;3,4,2,0,2,0)$.
\end{definition}

Next, we describe the genericity condition that the graphs $\Gamma_g$ must satisfy for
the algorithm to work correctly. The following lemma, which motivates genericity, is Example~2.1 in
\cite{cools2010tropical} and is crucial both for their algorithm and for the proof of our own Theorem~\ref{th:result}.
\begin{lemma}[Recentering $v_i$-reduced divisors on $\Gamma_g$]\label{lem:park}
   Let $c$ be a $v_{i-1}$-reduced divisor.
   Suppose that $c$ has $k$ chips on $v_{i-1}$, and that the chip on
   $\gamma_i\setminus\{v_{i-1}\}$ is a counter-clockwise distance $x_i$ from
   $v_{i-1}$ (with $x_i=0$ if there is no chip).

   Then $c$ is equivalent to the $v_i$-reduced divisor $c'$ which agrees with $c$ everywhere
   outside $\gamma_i$, and which restricts on $\gamma_i$ according to the following cases: 
   \begin{enumerate}
	   \item If $x_i=0$ and $k\geq 1$, then $c'$ has $k-1$ chips on $v_i$ and one chip that is
		   $(k-1)m_i$ away clockwise from $v_{i-1}$;  
	   \item if $x_i\not\equiv (k+1)m_i\bmod (\ell_i+m_i)$ then $c'$ has $k$ chips
		   on $v_i$ plus a chip that is $km_i-x_i$ away clockwise from
		   $v_{i-1}$;
	   \item if $x_i\equiv (k+1)m_i\bmod (\ell_i+m_i)$, then $c'$ has $k+1$ chips on
		   $v_{i-1}$.
   \end{enumerate}
\end{lemma}

Note that the first case can result in all $k$ chips being moved from
$v_{i-1}$ to $v_i$ if and only if $(k-1)m_i\equiv\ell_i\equiv-m_i\bmod(\ell_i+m_i)$, which is
the same as requiring that $km_i$ is an integer multiple of $\ell_i+m_i$,
i.e. that $\frac{\ell_i}{m_i}=\frac {k-n}n$ for some positive integer $n$. If $\deg(c)\leq 2g-2$,
however, then certainly $k\leq 2g-2$, so this can only happen if $\ell_i/m_i$ can be
written as the ratio of two integers with sum at most $2g-2$. Thus, the notion of genericity that we define below ensures that the first case of the lemma never results in all $k$ chips being moved from $v_{i-1}$ to $v_i$. 

\begin{definition}
    \label{def:generic}
    We say that $\Gamma_g$ is \emph{generic} if none of $\ell_i/m_i$ can
	be written as the ratio of two positive integers with sum at most $2g-2$.
\end{definition}

\begin{example}
The graph $\Gamma_6$ of Example~\ref{ex:running}
with edge lengths $m_1=m_2=\dots=m_6=1$ and $\ell_1=\ell_2=\dots=\ell_6=10$ is
generic since $\ell_i/m_i=10/1$ and $10+1 > 2\cdot g-2=2\cdot6-2=10$.
\end{example}

\subsection{The algorithm of Cools et al., and the bijection $\phi$}\label{subsec:alg}

We proceed with describing the algorithm of Cools et al.\ Its input, as indicated in the
previous subsection, are sequences $(d_0;x_1,\ldots,x_g)$ with $d_0\in\ZZ$ and $x_i\in\RR$
that encode $v_0$-reduced divisors according to the scheme
of Definition~\ref{def:seq}. Next, we define the objects which will constitute the algorithm's
output, and afterward we specify the algorithm itself. 

\begin{definition}
    Fix a positive integer $r$. Define the \emph{Weyl chamber} $\mathcal
	C\subset\ZZ^r$ to consist of those points $p$ for which
	$p(1)>p(2)>p(3)>\dots>p(r)>0$.   
	%
	Then an \emph{$r$-dimensional lingering lattice path} is a sequence
	of points $p_0,p_1,\ldots,p_g$ in the Weyl chamber such that:
	\begin{enumerate}
		\item $p_0=(d_0,d_0-1,d_0-2,\ldots,d_0-(r-1))$ for some positive integer $d_0$;
		\item For any $1\leq i\leq g$, we have that each step $p_i-p_{i-1}$ is either one
			of the standard basis vectors $e_i$ for $\ZZ^r$, the negative
			diagonal $(-1,-1,\ldots,-1)$, or zero;
	\end{enumerate}
	We denote the set of $r$-dimensional lingering lattice paths by $\LLP_r$, and
	the set of all lingering lattice paths by $\LLP$. Note that, for the
sake of brevity, our notion of lingering lattice path is more
restrictive than that in \cite{cools2010tropical} and captures only the objects
pertinent to our algorithmic reinterpretation of their results.

	An \emph{$r$-dimensional non-lingering lattice path} is a lingering lattice
	path in which:
	\begin{enumerate}
		\item No step ``lingers,'' in the sense that $p_i-p_{i-1}$ is never $0$;
		\item The total number of steps in the $(-1,-1,\dots, -1)$ direction, which we abbreviate as the 
		$r+1^\text{st}$ direction, equals
			the number of steps in the first direction.
        \item The integer $d_0$ as defined in the definition of lingering lattice path is equal to $r$, so that $p_0 = (r,r-1,\ldots,1)$.
	\end{enumerate}
	We denote the set of $r$-dimensional non-lingering lattice paths by $\NLLP_r$, and
	the set of all non-lingering lattice paths by $\NLLP$.
\end{definition}
\begin{remark}
\label{remark:pathvisualization}
It is convenient to visualize the sequences $(p_0,p_1,\ldots,p_g)\subset\ZZ^r$ as a collection of $r$ piecewise-linear
paths in $\RR^2$ defined by requiring that for each $r\geq j\geq 1$, the points $(0,p_0(j)),(1,p_1(j)),\ldots,(g,p_g(j))$ are
all cusps of the $j^\text{th}$ path. An example of this visualization is given in Figure \ref{fig:nonlingering}. Note that
because the points of lingering lattice paths are in the Weyl chamber, they are non-intersecting with the points $(0, p_0(j)), \dots , (g, p_g(j))$ above $((0,p_0(k)),\dots , (g, p_g(k))$ for $j < k$.
\end{remark}

\begin{figure}[htbp]
    \centering
    \includegraphics{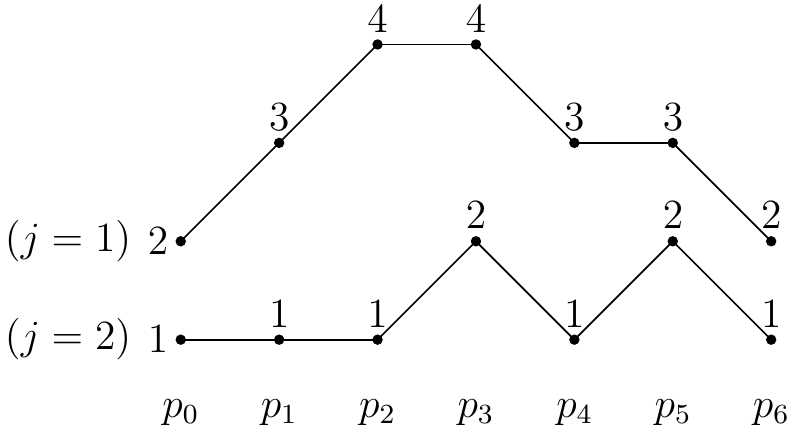}
    \caption{A $2$-dimensional non-lingering lattice path}
    \label{fig:nonlingering}
\end{figure}

Having described the output of the algorithm, we now specify the algorithm itself and collect its combinatorial
properties in the following theorem, which is a restatement of \cite[Theorem 1.4]{cools2010tropical} and its proof.

\begin{algorithm} \label{alg} If $c$ is a $v_0$-reduced divisor given by
	$(d_0;x_1,\ldots,x_g)$, then compute $\rho_r(c)=(p_0,p_1,\ldots,p_r)$
	according to the procedure:
	\begin{enumerate}
		\item $p_0=(d_0,d_0-1,\ldots,d_0-(r-1))$;
		\item $p_i-p_{i-1}=
		   \begin{cases}
			   (-1,\ldots,-1)&\text{if }x_i=0,\\ ~&~ \\
				e_j&
					\begin{aligned}
						&\text{if }x_i\equiv (p_{i-1}(j)+1)m_i\bmod \ell_i+m_i \\
						&\text{and both }p_{i-1},p_{i-1}+e_j\in\mathcal C,
					\end{aligned}\\ ~& ~ \\
			   0&\text{otherwise.}
		   \end{cases}$
	\end{enumerate}
\end{algorithm}

\begin{theorem}\label{th:cools}
	Suppose that for each positive integer $r$ and generic $\Gamma_g$, the above algorithm, Algorithm \ref{alg}, is well-defined and guaranteed to terminate, and returns a map
	$\displaystyle\rho_r\colon\Pic^{\leq2g-2}(\Gamma_g)\to\ZZ^r$  for generic $\Gamma_g$.  Then, identifying each element of $\Pic^{\leq 2g-2}(\Gamma_g)$ with its
    $v_0$-reduced divisor $c$, we have that $\rho_r(c)$ satisfies the
    following properties:
	\begin{enumerate}
		\item $c\in\Pic^{\leq2g-2}(\Gamma_g)$ is of rank at least $r$ if and
			only if $r\leq\max\{-1, d_0\}$ and $\rho_r(c)$ is a lingering
			lattice path in $\ZZ^r$, that is, $\rho_r(c)$ is in the Weyl chamber.  In particular, if $\rho_r(c)$ is a non-lingering lattice path, then $r$ must be at most the rank of $c$ since $d_0 =r $ so $r \leq \max\{-1, d_0\}$ clearly holds.
		\item if $\rho\colon\Pic^{\leq2g-2}(\Gamma_g)\to\LLP$ is defined by
			$\rho(c)=\rho_r(c)$ if $c$ has rank $r$, then there exists a map
			$\alpha\colon\NLLP\to\Pic(\Gamma_g)$ with $\rho\circ\alpha = 1$, whose
			image consists of all rank $r$ degree $d$ elements of
			$\Pic(\Gamma_g)$ such that $(g-d+r)(r+1)=g$.
	\end{enumerate}
\end{theorem}

\begin{example}
    The non-lingering lattice path indicated in Figure \ref{fig:nonlingering} is the result of applying
    the above algorithm to the $v_0$-reduced divisor of Example \ref{ex:running}, which was given by
    $(2;3,4,2,0,2,0)$ on the graph $\Gamma_6$ with edge lengths $\ell_i=10$, $m_i=1$.
    
    Since $m_i=1$, applying the algorithm is particularly simple since $p_i-p_{i-1}=e_j$
    if $x_i\equiv p_{i-1}(j)+1\bmod(\ell_i+m_i)$, i.e. if $x_i$ is one more than $p_{i-1}(j)$ for
    some $j$, we increase the $j^\text{th}$ path. If $x_i$ is zero, we decrease all paths, and if
    $x_i$ is neither $0$ nor one more than $p_{i-1}(j)$ for some $j$, we linger.
    
    Thus, starting with $p_0=\binom{2}{1}$, the fact that $x_1=3$ and then $x_2=4$ give us that
    $p_1=\binom{3}{1}$ and then $p_2=\binom{4}{1}$. Next, the fact that $x_3=2$ and then $x_4=0$ give
    $p_3=\binom{4}{2}$ and then $p_4=\binom{3}{1}$. Finally, $x_5=2$ and $x_6=0$ give $p_5=\binom{3}{2}$ 
    and $p_6=\binom{2}{1}$.
\end{example}

With the above theorem, Cools et al.\ construct their bijection $\phi$ from the map $\alpha$ and a
bijection $\beta$ between rectangular standard Young tableaux and non-lingering lattice paths.

\begin{definition}
    Let $\SYT(n^m)$ be the set of rectangular $m\times n$ \emph{standard Young
		tableaux}, that is, $m\times n$ matrices $(a_{ij})$ whose entries are
	all the integers from $1$ to $mn$ such that $a_{i,j+1},a_{i+1,j}>a_{i,j}$.
	Let $\RSYT$ be the set of all rectangular standard Young tableaux. See \cite{stanley2008evacuation}
    for more details and background.
\end{definition}

\begin{proposition}[Proved in the proof of Theorem $1.4$ of \cite{cools2010tropical}] \label{prop:RST}
    Suppose that $P=(p_0,p_1,\ldots,p_g)$ is an $r$-dimensional non-lingering
	lattice path such that $g=(g-d+r)(r+1)$. We can fill a $(g-d+r)\times(r+1)$ standard Young tableau with
    the numbers from $1$ to $g$ as follows. Starting with
	$1$ and ending with $g$, place the number $i$ in the topmost free spot of
	the $j^\text{th}$ column if $p_i-p_{i-1}=e_j$, and in the $r+1^\text{st}$
	column if $p_i-p_{i-1}=(-1,\ldots,-1)$. Furthermore, all $(g-d+r)\times(r+1)$
	rectangular standard Young tableaux can be obtained in this way from one and
	only one non-lingering lattice path.
\end{proposition}

\begin{remark}
In the case $r=0$, Proposition \ref{prop:RST} still holds.  In this case, the equality $g=(g-d+r)(r+1)$ implies that $d=0$ and for each positive integer $g$, the zero divisor is indeed the unique effective divisor of degree and rank zero.  This corresponds to the empty path and to the unique $g \times 1$ standard Young tableau.  
\end{remark}

\begin{example}\label{ex:tab}
    Starting with the $r=2$-dimensional non-lingering lattice path of Figure~\ref{fig:nonlingering},
    we construct the standard Young tableaux as follows. From $p_0=\binom{2}{1}$ to $p_1=\binom{3}{1}$
    we have an increase in the first path so we put a $1$ in the first column. Going to $p_2=\binom{4}{1}$
    and then to $p_3=\binom{4}{2}$ we have increases in the first and then the second path, which means we
    put $2$ in the first column and $3$ in the second. Next we have a decrease to $p_4=\binom{3}{1}$ which
    means we put $4$ in column $r+1=3$. Continuing, we obtain the tableau:
    
    \[
    T = 
    \begin{tabular}{|r|r|r|}
        \hline
		1 & 3 & 4 \\
		\hline
		2 & 5 & 6 \\
		\hline
	\end{tabular}
    \]
\end{example}

\begin{definition}
    \label{def:phi}
	Let $\beta\colon\RSYT\to\NLLP$ be the inverse of the
	above-described bijection.
    
    Define the injective map $\phi\colon\RSYT\to\Pic(\Gamma_g)$ by
	$\phi=\alpha\circ\beta\colon\RSYT\to\Pic(\Gamma_g)$. Then the map $\phi$ bijects rectangular standard Young tableaux
    onto rank $r$ degree $d$ $v_0$-reduced divisors such that $(g-d+r)(r+1)=g$. 
\begin{figure}[h]\label{fig:phi}
    \[
        \xymatrix{
    		\Pic(\Gamma_g)&
				\ar@{_(->}[l]_-\alpha
				\NLLP
					&
					\ar@{_(->}[l]_-\beta
					\ar@{.>}@/_1.5pc/[ll]_-\phi
					\RSYT
          }
    \]
\caption{The map $\phi$}
\end{figure}
\end{definition}

\section{Evacuation and Reflection}\label{sec:res}
In this section we state and prove our original result
regarding how evacuation on standard Young tableaux acts on $v_0$-reduced
divisors on the generic graphs $\Gamma_g$ under the bijection $\phi$ of Definition~\ref{def:phi}
discovered by Cools et al.\ in \cite{cools2010tropical}. 

\begin{definition}
    We consider an involution $\ev\colon\RSYT\to\RSYT$, called
    \emph{evacuation}, sending $T$ in $\RSYT$ to $\ev(T)=(b_{ij})$ where
	$b_{i,j}=(mn+1-a_{m+1-i,n+1-j})$. Evidently, evacuation preserves the
	dimensions of tableaux.
\end{definition}
Geometrically, $\ev(T)$ can be pictured as rotating the rectangular standard
Young tableaux $180^\circ$ and flipping the entries according to the rule $i\to
mn+1-i$. For more details on evacuation of tableaux, introduced by Sch\"utzenberger \cite{schutz}, we direct the interested reader to
the wonderful survey in \cite{stanley2008evacuation} by Richard Stanley. 

\begin{definition}
	Suppose that $\Gamma_g$ is as in Definition~$\ref{def:gamma}$, i.e. the
	concatenation of $g$ circles $\gamma_i$ of circumferences $\ell_i+m_i$,
	along with designated vertices $\{v_i\}_{0\leq i\leq g}$ such that
	$v_{i-1},v_i\in\gamma_i$ and $\ell_i$ and $m_i$ are the lengths of the two
	arcs joining $v_{i-1}$ and $v_i$ in $\gamma_i$. Define the \emph{reflection}
    $\Gamma'_g$ of $\Gamma_g$ to be the same compact metric graph as $\Gamma_g$,
    but with designated vertices $v'_i=v_{g+1-i}$. Note that this gives $\ell'_{g+1-i}=\ell_i$
    and $m'_{g+1-i}=m_i$.

	Identifying $\Pic(\Gamma_g)$ with the set of $v_0$-reduced divisors, we define the \emph{reflection}
	$\sigma\colon\Pic(\Gamma_g)\to\Pic(\Gamma'_g)$ to be given by the rule that
	if $c$ is a $v_0$-reduced divisor on $\Gamma_g$, then $\sigma(c)$ is the
    $v_g=v'_0$-reduced divisor equivalent to $c$ considered as a divisor on $\Gamma'_g$.
\end{definition}

\begin{theorem}\label{th:result}
	Suppose that $\Gamma_g$ is generic and suppose $(g-d+r)(r+1) =g$. Let $\phi=\alpha\circ\beta\colon\RSYT\to\Pic(\Gamma_g)$
    be the map of Definition~\ref{def:phi} due to Cools et al.\ that bijects $(g-d+r)\times (r+1)$ rectangular standard Young tableaux
    to rank $r$ degree $d$ $v_0$-reduced divisors on $\Gamma_g$.
    Let $\phi'=\alpha'\circ\beta\colon\RSYT\to\Pic(\Gamma'_g)$ be the analogous map
    from rectangular standard Young tableaux to $v_0$-reduced divisors on the reflection $\Gamma'_g$. 
    Then evacuation on tableaux corresponds to reflection in the sense that the
	following diagram commutes:
    \[
    	\xymatrix{
			\Pic(\Gamma_g)
			\ar@{->}[d]^\sigma
				&
				\ar@{_(->}[l]_-\alpha
				\NLLP
					&
					\ar@{_(->}[l]_-\beta
					\ar@{.>}@/_1.5pc/[ll]_-\phi
					\RSYT
					\ar@{->}[d]^\ev &\\
			\Pic(\Gamma'_g)
				&
				\ar@{_(->}[l]_-{\alpha'}
				\NLLP
					&
					\ar@{_(->}[l]_-\beta
					\ar@{.>}@/^1.5pc/[ll]_-{\phi'}
					\RSYT&
		}
	\]
\end{theorem}

The remainder of this section is devoted to the proof this theorem. In Subsection~\ref{subsec:1}, we introduce our useful notation
for elements of $\Pic(\Gamma_g)$ in the image of $\phi$, and give formulas for $\alpha$, $\beta$, $\phi$, and $\phi'\circ\ev$.

Once the notation is laid out, the heart of this proof is Proposition
\ref{prop:sigmaalpha}, where we characterize non-lingering lattice paths $p$ based on the sequence of differences $p_i - p_{i-1}$. These differences are characterized by three cases: a change in the coordinate $j$ where $j=0$, $0<j<r$, or $j = r$. As we will show, this
characterization is symmetric in $j$ and $r-j$, so that reading the path $p'$ obtained from reading $c$ backwards must change in the $r-j$ direction, exactly as evacuation would suggest. 

In Subsection~\ref{subsec:2} we give a formula for $\sigma\circ\phi$, which we show is the same as the formula for $\phi'\circ\ev$, thus proving Theorem~\ref{th:result}.

\subsection{Notation, and formulas for $\phi=\alpha\circ\beta$ and $\phi'\circ\ev$}\label{subsec:1}

\begin{proposition}[Formula for $\alpha$]
	If $P$ is an $r$-dimensional non-lingering lattice path, then
	$\alpha(P)=(r;x_1,\ldots,x_g)$ can be computed via the formula:
	\begin{align*}	
		x_i&=\begin{cases}
			0&\text{if }p_i-p_{i-1}=(-1,-1,\ldots,-1)\\
			(p_{i-1}(j)+1)m_i&\text{if }p_i-p_{i-1}=e_j
			\end{cases}
	\end{align*}
    \begin{proof}
    Immediate from the definition of $\alpha$ in Theorem~\ref{th:cools} as a map that inverts
    $\rho$ and the definition of $\rho$.
    \end{proof}
\end{proposition}

\begin{notation}
	If $c$ is a $v_0$-reduced divisor on $\Gamma_g$ such that
	$(g-d+r)(r+1)=g$, then we can describe $c$ by a sequence
	$(d_0;\underline x_0, \underline x_1,\ldots,\underline x_g)$ where
	$\underline x_i m_i$ is the counter-clockwise distance of the single chip
	on $\gamma_i$ from $v_i$, if such a chip exists, and $\underline x_i=0$ if
	the chip does not exist. Again, this description is unique up to the modular
	equivalences $\underline x_i\bmod((\ell_i+m_i)/m_i)$. Note that if $m_i = 1$,
    we get $\underline x_i\equiv x_i- 1\bmod(\ell_i+m_i)$ since $x_i$ measures the
    counter-clockwise distance modulo $l_i + 1$ to the chip, starting from $v_{i-1}$,
    while $\underline x_i$ measures the distance starting from $v_{i}$ instead.

	Using the $\underline x_i$ notation, we can thus rewrite our formula for $\alpha$ from above as:
		\begin{align}\label{eq:alpha}
			\underline x_i&=
				\begin{cases}
					0&\text{if }p_i-p_{i-1}=(-1,-1,\ldots,-1)\\
					p_{i-1}(j)&\text{if }p_i-p_{i-1}=e_j
				\end{cases}
		\end{align}	
\end{notation}

\begin{example}\label{ex:notation}
    Consider once again the $v_0$-reduced divisor $(2;3,4,2,0,2,0)$ on $\Gamma_6$ with $\ell_i=10$, $m_i=1$ from Example \ref{ex:running}.
    Since the algorithm gives a non-lingering lattice path, we know that it satisfies $(g-d+r)(r+1)=g$.
    In particular, it is of degree $6$ and rank $2$, while the genus is $6$.
    
    Hence in the new notation its
    sequence is $(2;\underline 2,\underline 3,\underline 1,\underline 0,\underline 1,\underline 0)$.
\end{example}

\begin{remark}
Even though the images of $\alpha$ and $\alpha'$ are usually different, in particular they map to  divisors on $\Gamma_g$ and $\Gamma'_g$ respectively, the images of $\phi$ and $\phi'$ agree as tuples $(r;\underline x_1,\ldots\underline x_r)$.  Technically the tuples in the image of $\phi$ is defined modulo $(\ell_i+m_i)/m_i$ and the image of $\phi'$ is defined modulo  $(\ell_{g+1-i}+m_{g+1-i})/m_{g+1-i}$, but since each $\ell_i+m_i$ is assumed to be greater than $2g-2$, if we assume each $m_i = 1$, then we can take the distances to satisfy $0\leq \underline{x}_i\leq (\ell_i+1)$ so that the moduli never come into play.  For other choices of $m_i$, the distance chosen might need to be larger than $(\ell_i+m_i)/m_i$, i.e. wrap around the loop $\gamma_i$, to agree with the values from (\ref{eq:alpha}).  For convenience, we will assume distances are chosen to be multiples of $m_i$ and agreeing with the values from (\ref{eq:alpha}) throughout the rest of this paper.
\end{remark}



\begin{proposition}[Formula for $\beta$]\label{prop:path}
	Suppose that $T$ is a rectangular $(g-d+r)\times(r+1)$ standard Young
	tableau and that $\beta(T)=P=(p_0,\ldots,p_g)$. Then for any $i$, $p_i$ can
	be computed according to the formula:
	\begin{align}\label{eq:beta}
		p_i &= p_0+\left(\begin{array}{c}
					l_1-l_{r+1}\\
					l_2-l_{r+1}\\
					\dots\\
					l_r-l_{r+1}\end{array}\right)
			= \left(\begin{array}{c}
				r+l_1-l_{r+1}\\
				r-1+l_2-l_{r+1}\\
				\dots\\
				1+l_r-l_{r+1}\end{array}\right)
	\end{align}
	where $l_s$ is the number of cells in the $s^\text{th}$ column of $T$ whose
	entries are at most $i$.
	\begin{proof}
		In the bijection of Cools et al.\ (see Proposition \ref{prop:RST}) 
		between the non-lingering lattice paths and the standard
		Young tableaux, a number $k\leq i$ is placed in column $j<r+1$ when
		$p_k-p_{k-1}=e_j$, i.e., when there has been an increase in the
		$j^\text{th}$ direction. The number $l_j$ hence counts the number of
		increases that have occurred in the $j^\text{th}$ direction by the
		$i^\text{th}$ step.
		
		On the other hand, a number $k\leq i$ is placed in column $r+1$
		when $p_k-p_{k-1}=(-1,-1,\ldots,-1)$, i.e., when there has been a decrease
		along all directions. Hence, $l_{r+1}$ counts the number of decreases
		that have occurred by the $i^\text{th}$ step. Knowing that we start with
		$p_0=(r,r-1,\ldots,1)$, it follows that $p_i(j)=p_0(j)+l_j-l_{r+1}$ and
		hence the proposition follows.
	\end{proof}
\end{proposition}


\begin{notation}
	If $T$ is a rectangular $m \times n$ standard Young tableau and $1 \leq i \leq mn$, then we let 
\begin{itemize}
\item[(i)] 	$l_r(i,T)$ denote the index of the row of $T$ (from top to bottom) containing $i$, 

\item[(ii)] $l_c(i,T)$ denote the index of the column (from left to right) containing $i$,

\item[(iii)] $L_{first}(i,T)$ denote the number of cells in the first 
        column whose entries are strictly less than $i$, and 
        
\item[(iv)] $L_{last}(i,T)$ 
        denote the number of cells in the last 
        column whose entries are strictly less than $i$.	
\end{itemize}
\end{notation}

Putting the formulas for $\alpha$ and $\beta$ together, we obtain a formula for
$\phi=\alpha\circ\beta$.

%
%

\begin{proposition}[Formula for $\phi$]\label{prop:x}
	Suppose that $T\in\SYT((r+1)^{g-d+r})$, and that
	$\phi(T)=\alpha\circ\beta(T)=c$ is described by $(r;\underline
	x_1,\ldots,\underline x_g)$.
	
	Then the $\underline x_i$'s can be computed according to the formula:
		\begin{align}\label{eq:phi}
			\underline x_i&= r + l_r(i,T) - l_c(i,T) - L_{last}(i,T)
		\end{align}
	using the above notation.
	\begin{proof}
		Let $\beta(T)=P=(p_0,\ldots,p_g)$ so that $\alpha(P)=c$.
		If the number $i$ is in the $l = l_c(i,T)^\text{th}$ column of $T$ for $1\leq
		l<r+1$, then $p_i-p_{i-1}=e_l$ and hence formula ($\ref{eq:alpha}$) for
		$\alpha$ gives us $\underline x_i=p_{i-1}(l)$, while formula
		($\ref{eq:beta}$) for $\beta$ gives us 
		$$p_{i-1}(l) 
		= (r+1-l) + (l_r(i,T)-1) - L_{last}(i,T)),$$ yielding formula (\ref{eq:phi}).
                Note that since $l_r(i,T)$ is the index of the row 
                containing $i$, it follows that $l_r(i,T) - 1$ is the number of 
                entries in column $l$ that are at most $i-1$. 		

		Otherwise, if the number $i$ is in the $r+1^\text{st}$ column of $T$,
		then $p_i-p_{i-1}=(-1,\ldots,-1)$.  
		We obtain $\underline x_i = 0$ in this case, and 
		since $l_c(i,T) = r+1$ and $l_r(i,T) = L_{last}(i,T)+1$, the proposed formula for $\phi$
		still holds.
	\end{proof}
\end{proposition}

%
%
%

\begin{example}
    Consider the tableau of Example~\ref{ex:tab}
    
    \[
    T=
    \begin{tabular}{|r|r|r|}
        \hline
    	1 & 3 & 4 \\
		\hline
		2 & 5 & 6 \\
		\hline
	\end{tabular}.
    \]
    
    Using the formula for $\phi$ from Proposition~\ref{prop:x}, we obtain the sequence $\phi(T)=(2;\underline 2,\underline 3,\underline 1,\underline 0,\underline 1,\underline 0)$
    agreeing with Example~\ref{ex:notation}. To compute $\underline x_5$, for example, we see that $i=5$ is in row $2 = l_r(5,T)$ and column $2=l_c(5,T)$.  Also, the number of cells in the last columns that are strictly less than $i=5$ is $L_{last}(5,T) = 1$.
    It follows that $\underline x_5=r + l_r(5,T) - l_c(5,T) - L_{last}(5,T) = 2 + 2 - 2 - 1 = 1$. 
\end{example}

Next, we obtain the formula for $\phi'\circ\ev$. 

\begin{proposition}[Formula for $\phi'\circ\ev$]
    \label{prop:phiprimeev}
	Suppose that $T$ is a $(g-d+r)\times(r+1)$ rectangular standard Young
	tableaux and that $\phi'\circ\ev(T)=\alpha'\circ\beta\circ\ev(T)=c'$ is
	described by $(r;\underline x'_1,\ldots,\underline x'_g)$.

	Then the $\underline x'_i$'s can be computed according to the formula:
		\begin{align}\label{eq:phi'}
			\underline x'_{g+1-i}&=j-1+l_1-l_j
		\end{align}	
	where $j=l_c(i,T)$ is the index of the column containing $i$ and $l_s$ is the number of cells in the $s^\text{th}$ column of $T$ whose
	entries are at most $i$.
	\begin{proof}
		Evacuation takes column $j=l_c(i,T)$ of $T$ to column $r+2-j$. Hence, if $l_j$ is
		the number of cells in column $j$ of $T$ whose entries are at
		most $i$, then $l_j$ is also the number of cells in column $r+2-j$ of
		$\ev(T)$ whose entries are at least $g+1-i$ (since evacuation also flips
		the values of the entries).

		If we let $\beta(\ev(T))=(p'_0,p'_1,\ldots,p'_g)$, then the $l_s$
		acquire the following significance. For $j\neq 1$, $l_j$ counts the
		number of increases in the $r+2-j^\text{th}$ direction from $p'_{g-i}$
		to $p'_g=(r,r-1,\ldots,1)$. Similarly, $l_1$ counts the number of
		decreases along all directions from $p'_{g-i}$ to
		$p'_g=(r,r-1,\ldots,1)$. 
		
		Given that the $l_j$ are counting steps in the $r+2-j^\text{th}$
		direction from $p'_{g-i}$ to $p'_g$, we obtain the following analogue of
		Proposition~$\ref{prop:path}$:
			\[
				p'_{g-i} = p'_g-\left(\begin{array}{c}
						l_{r-1}-l_1\\
						l_r-l_1\\
						\dots\\
						l_2-l_1\end{array}\right)
				= \left(\begin{array}{c}
					r+l_1-l_{r+1}\\
					r-1+l_1-l_r\\
					\dots\\
					1+l_2-l_1\end{array}\right).
			\]

		Next we obtain the analogue of Proposition~$\ref{prop:x}$ using exactly
		the same argument. Suppose that $i$ is in the $j^\text{th}$ column of
		$T$. Then we have that $g+1-i$ is in the $r+2-j^\text{th}$ column of
		$\ev(T)$. If $j\geq 2$, then $p'_{g+1-i}-p'_{g-i}=e_{r+2-j}$ and hence
		$x'_{g+1-i}=p'_{g-i}(r+2-j)=j-1+l_1-l_j$ where $l_s$ is the number of
		cells in the $s^\text{th}$ column of $T$ whose entries are at most $i$.
		
		Otherwise, if $j=1$, we have that $g+1-i$ is in the $r+1^\text{st}$
		column of $\ev(T)$, which means that $p'_{g+1-i}-p'_{g-i}=(-1,\ldots,-1)$
		and $x'_{g+1-i}=0$. Since $0=1-1+l_1-l_1$, the formula holds in both
		cases.
	\end{proof}
\end{proposition}

\begin{example}\label{ex:ev}
    Consider the tableau of Example~\ref{ex:tab} and its associated non-lingering lattice path $P' = \beta(\ev(T))$ shown in Figure~\ref{fig:pathreflected}.
    
    \[T = 
    \begin{tabular}{|r|r|r|}
        \hline
        1 & 3 & 4 \\
		\hline
		2 & 5 & 6 \\
		\hline
	\end{tabular}
    \hspace{1in}
    \ev(T) = 
        \begin{tabular}{|r|r|r|}
        \hline
        1 & 2 & 5 \\
    	\hline
		3 & 4 & 6 \\
		\hline
	    \end{tabular}
    \]
    
    \begin{figure}[htbp]
        \centering
        \includegraphics{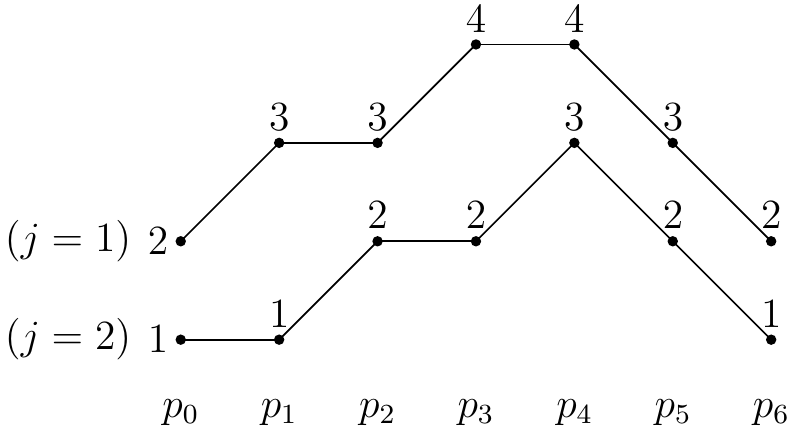}
        \caption{The $2$-dimensional non-lingering lattice path $P'$ associated to $\ev(T)$.}
        \label{fig:pathreflected}
    \end{figure}
    
    This example will demonstrate the shortcut formula for $\phi'\circ\ev$  of Proposition~\ref{prop:phiprimeev}. To compute $\underline x'_3$, for example, we see that $g+1-i=3$ implies $i=7-3=4$, and that $4$ is in column $3=j$.
    The number $l_3$ of cells in the third column that are at most $i=4$ is $1$, and the number of cells $l_1$ in the
    first column that are at most $i=4$ is $2$. It follows that $\underline x_3=j-1+l_1-l_j=3-1+2-1=3$. Doing this for
    each $\underline x_i$, we obtain the sequence
    $c'=\alpha'(P')=(2;\underline 2,\underline 1,\underline 3,\underline 2,\underline 0,\underline 0)$. In Example~\ref{ex:postshifting}, we will show that this sequence corresponds to the reflection $\sigma(c')$ of the $v_0$-reduced divisor
    $(2;\underline 2,\underline 3,\underline 1,\underline 0,\underline 1,\underline 0)$ from Example~\ref{ex:notation}.

\begin{figure}[htbp]
    \centering
    \includegraphics{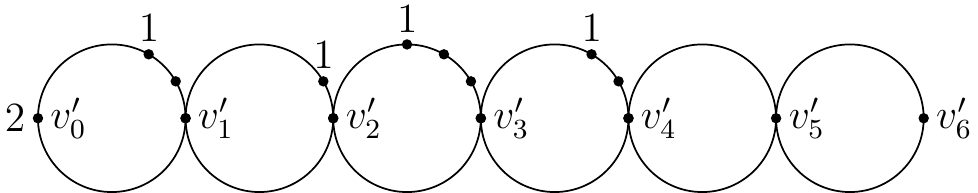}
    \caption{The proposed reflection of the $v_0$-reduced divisor from Example \ref{ex:running}}
    \label{fig:runningreflection}
\end{figure}
\end{example}

\subsection{Formula for $\sigma\circ\alpha$ and proof of Theorem~\ref{th:result}}\label{subsec:2}
We compute a formula for $\sigma\circ\alpha$ using the following lemma.

\begin{lemma}
    \label{lem:chipshifting}
	Let $c$ be a $v_0$-reduced divisor of non-negative degree at most $2g-2$, and for $i \geq 0$ let $c_i$ be the
	$v_i$-reduced divisor that is equivalent to $c$. Then if
	$\rho(c)=P=(p_0,p_1,\ldots,p_g)$, we have that $c_i(v_i)=p_i(1)$.
	\begin{proof}
		Trivially, we have that $p_0(1)=d_0=c_0(v_0)$.
		
		Next, suppose inductively that $p_{i-1}(1)=k=c_{i-1}(v_{i-1})$. By
		Lemma~\ref{lem:park} we know that $c_i(v_i)-c_{i-1}(v_{i-1})=-1,0,$ or
		$1$ depending on whether $x_i=0$, $x_i\equiv(k+1)m_i\bmod(\ell_i+m_i)$,
		or $x_i\not\equiv(k+1)m_i\bmod(\ell_i+m_i)$. But since $p_{i-1}(1)=k$,
		these are precisely the conditions for $p_i-p_{i-1}=(-1,-1,\ldots,-1),
		0$, or $e_1$, and hence for $p_i(1)-p_{i-1}(1)=-1,0,$ or $1$.
	\end{proof}
\end{lemma}

\begin{proposition}[Formula for $\sigma\circ\alpha$]\label{prop:sigmaalpha}
	Suppose that $c\in\Pic(\Gamma_g)$ is such that $c=\alpha(P)$ for some
	non-lingering lattice path $P=(p_0,p_1,\ldots,p_g)$, i.e. that $c$ is of rank
	$r$ and degree $d$ such that $(g-d+r)(r+1)=g$, and further that $c$ is
	described by a sequence $(r;\underline x_1,\ldots,\underline x_g)$. 
	
	Then $\sigma(c)$ is also of rank $r$ and degree $d$ such that
	$(g-d+r)(r+1)=g$, and hence can be described by a sequence $(r;\underline
	x'_1,\ldots,\underline x'_g)$. The $\underline x'_i$'s can be computed from the $\underline x_i$'s
	using the formula:
		\begin{align}\label{eq:relate}
			\underline x'_{g+1-i}&=\max\{p_{i-1}(1)-\underline x_i-1,0\}.
		\end{align}
	\begin{proof}
		Note that under the reflection that takes $\Gamma_g$ to
		$\Gamma'_g$, the counter-clockwise distances from $v_i$ in the loop
		$\gamma_i$ are sent to clockwise distance from $v'
		_{1+g-i}=v_i$ in the loop $\gamma'_{g+1-i}=\gamma_i$. Hence, the $\underline x'_i$'s
		in the sequence $(r; \underline x'_1,\ldots, \underline x'_g)$ which describes $\sigma(c)$, the
		$v_g=v'_0$-reduced divisor equivalent to $c$, can be
		interpreted as both the counter-clockwise distance from $v'_i$ in the
		loop $\gamma'_i$, and as the clockwise distances from $v_{1+g-i}$ of the
		single chip on the loop $\gamma_{1+g-i}$.

		These clockwise distances, however, are determined by successively
		computing the equivalent $v_1$-reduced divisor, then the equivalent $v_2$-reduced
        divisor and so on until the equivalent $v_g$-reduced divisor. Lemma~\ref{lem:park} applies and gives
        us the following.
		
		Suppose that $k$ is the number of chips on $v_{i-1}$ of the $v_{i-1}$-reduced divisor
        $c_{i-1}$ that is equivalent to $c$. Then combining Lemma~\ref{lem:park} with our 
        notation for $\underline x_i$, we obtain:

		\begin{enumerate}
			\item if $\underline x_i$ is $0$, i.e. if there is no chip on
				$\gamma_i\setminus\{v_{i-1}\}$ in $c_{i-1}$, then there will be one chip on
				$\gamma_i\setminus\{v_i\}$ that is a clockwise distance
				$(k-1)m_i$ away from $v_{i-1}$ in $c_i$;
			\item if $\underline x_im_i\not\equiv km_i\bmod (l_i+m_i)$, then
				there is one chip on $\gamma_i\setminus\{v_i\}$  in $c_{i-1}$ that is a
				clockwise distance $km_i-(\underline x_i+1)m_i$ away from
				$v_{i-1}$ in $c_i$;
			\item if $\underline x_im_i\equiv km_i\bmod(l_i+m_i)$, then there
				are no chips left on $\gamma_i\setminus\{v_i\}$ in $c_i$.
		\end{enumerate}
		
		Now, $k$ is of course $p_{i-1}(1)$ by the previous lemma. Hence, the
		formula $\underline x'_{g+1-i}=\max\{p_{i-1}(1)-\underline x_i-1,0\}$
		holds for each of the three cases above. 
	\end{proof}
\end{proposition}

\begin{example}
    \label{ex:postshifting}
    Consider the $v_0$-reduced divisor
    $c = (r;\underline x_1,\ldots,\underline x_g)=(2;\underline 2,\underline 3,\underline 1,\underline 0,\underline 1,\underline 0)$ on
    the graph $\Gamma_6$ from Example~\ref{ex:notation}. Its associated non-lingering lattice path is the one from
    Figure~\ref{fig:nonlingering} with top path $(p_0(1),p_1(1),\ldots,p_g(1))=(2,3,4,4,3,3,2)$.
    
    Figure \ref{fig:chipshifting} illustrates the process described in the proof of Proposition~\ref{prop:sigmaalpha} of successively computing the $v_i$-reduced divisors $c_i$ equivalent to $c$. Looking at the top path $(2,3,4,4,3,3,2)$, the figure also illustrates the claim of Lemmas~\ref{lem:park} and \ref{lem:chipshifting}.
\begin{figure}[h!tbp]
\centering
\subfigure{\includegraphics{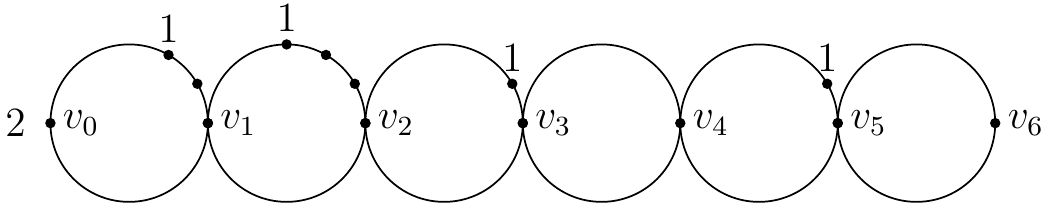}}
\subfigure{\includegraphics{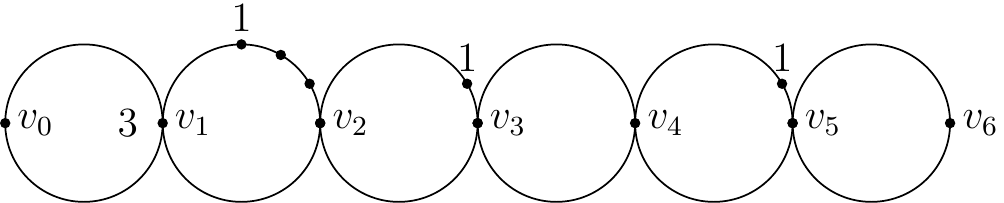}}
\subfigure{\includegraphics{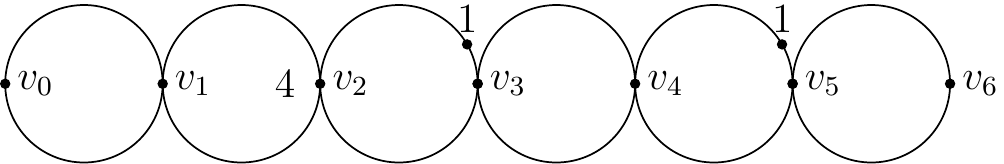}}
\subfigure{\includegraphics{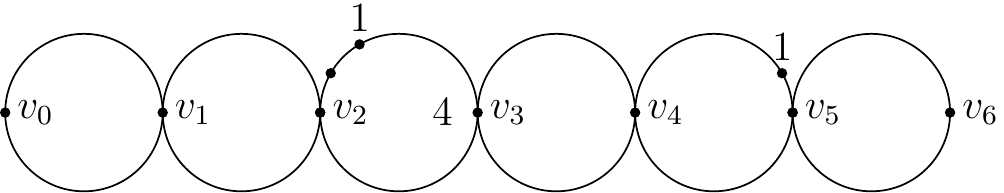}}
\subfigure{\includegraphics{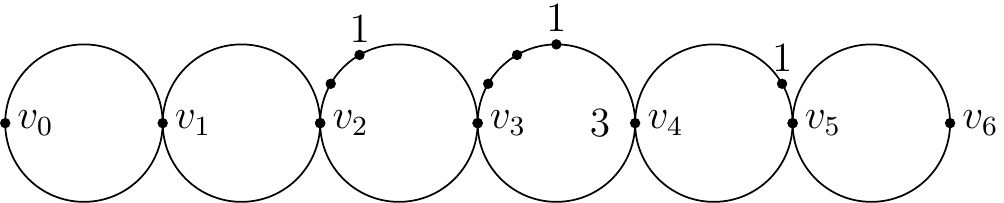}}
\subfigure{\includegraphics{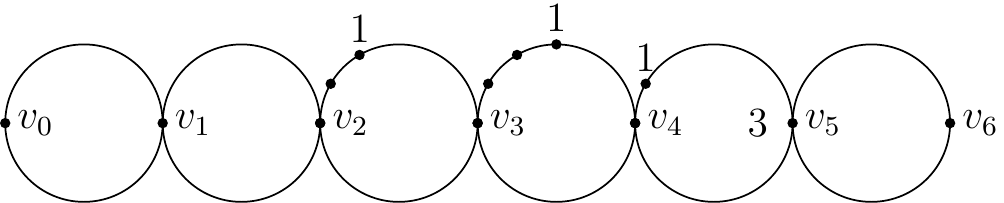}}
\subfigure{\includegraphics{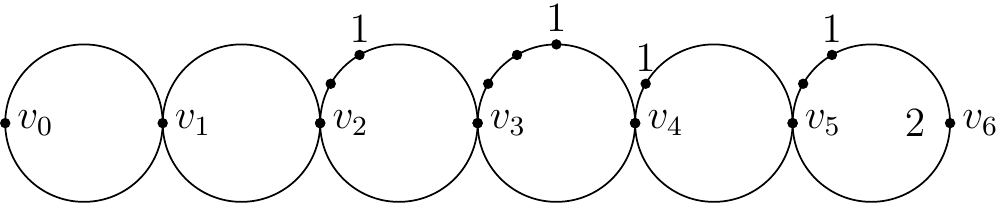}}
\begin{flushleft}
We reflect by setting $v_i = v'_{g+1-i}$, and obtain the configuration illustrated in Figure~\ref{fig:runningreflection}:
\end{flushleft}
\subfigure{\includegraphics{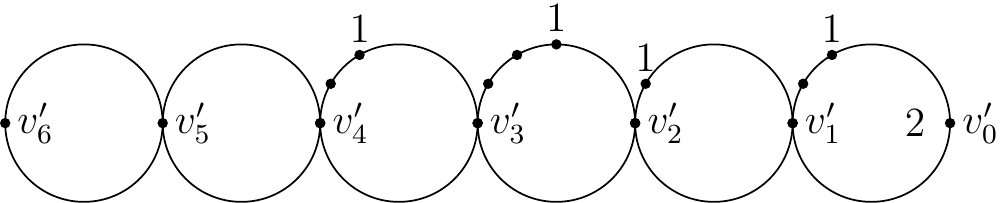}}
\subfigure{\includegraphics{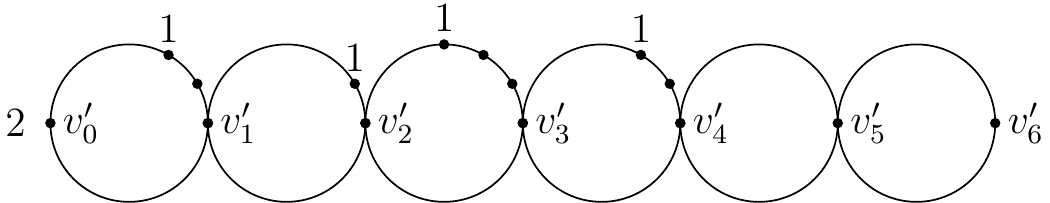}}
\caption{Successively computing $v_i$-reduced divisors}
\label{fig:chipshifting}
\end{figure}
    
    To illustrate Proposition~\ref{prop:sigmaalpha}, note that if we subtract the sequence of $\underline x_i$'s in $c=(2;\underline 2,\underline 3, \underline 1, \underline 0, \underline 1, \underline 0)$ from the sequence of $p_{i-1}(1)$'s for $1\leq i\leq g$, which is $(2,3,4,4,3,3)$, we obtain $(0,0,3,4,2,3)$.
    Subtracting a further $1$ from everything and reversing, we get $(2,1,3,2,-1,-1)$. Taking the maximum with $0$ and putting the rank $r=2$
    in front, we obtain the $v'_0$-reduced divisor on $\Gamma'_g$ given by
    $(r;\underline x'_1,\ldots,\underline x'_g)=(2;\underline 2, \underline 1,\underline 3,\underline 2,\underline 0,\underline 0)$, which is what Figure~\ref{fig:chipshifting} also produces.
    
    In accordance with our theorem, this is exactly the same $v'_0$-reduced divisor as the one obtained from evacuating the
    tableau in Example~\ref{ex:ev}.
\end{example}

\begin{proof}[Proof of Theorem~$\ref{th:result}$]
	We need to show that $\phi'\circ\ev=\sigma\circ\phi$, and we have so far
	formulas for $\phi'\circ\ev$, $\sigma\circ\beta$, and $\alpha$ where
	$\phi=\beta\circ\alpha$.

	We proceed to obtain a formula for $\sigma\circ\phi$, which we then reduce
	to the formula for $\phi'\circ\ev$.

	Suppose that $T$ is a rectangular standard Young tableaux. Let
	$\displaystyle\beta(T)=P=(p_0,\ldots,p_g)$,
	$\displaystyle\phi(T)=\alpha\circ\beta(T)=\alpha(P)=(r;\underline x_1,\ldots,\underline
	x_g)$ and
	$\sigma\circ\alpha(P)=(r;\underline x'_1,\ldots,\underline x'_g)$. We have
	established that:
	\begin{enumerate}
		\item $\underline x'_{g+1-i}=\max\{p_{i-1}(1)-\underline x_i-1,0\}$ by
			the formula~\eqref{eq:relate} for $\sigma\circ\alpha$.
%
%
		\item $p_{i-1}(1)=r+L_{first}(i,T)-L_{last}(i,T)$ by the formula~\eqref{eq:beta} for $\beta$ 

(noting that $i$ is not necesarrily in column $1$).
%
%
		\item $\underline x_i=r + l_r(i,T) - l_c(i,T) - L_{last}(i,T)$ by the formula~\eqref{eq:phi} for $\phi$.
	\end{enumerate}
	where the notation $L_{first}(i,T)$, $L_{last}(i,T)$, $l_r(i,T)$, and $l_c(i,T)$ is as above.
        Thus
		\begin{eqnarray*}
			\underline x'_{g+1-i}&=&\max\{(r+L_{first}(i,T)-L_{last}(i,T))\\
			~&~& ~ \hspace{0.43in} -(r + l_r(i,T)-l_c(i,T)-L_{last}(i,T))-1,0\} \\
			&=&\max\{L_{first}(i,T) - l_r(i,T)+l_c(i,T)-1,0\}.
	        \end{eqnarray*}
    Note that if $l_c(i,T) = 1$, then the number of cells in column $1$ whose entries 
are strictly less than $i$ is one less than the index of the row containing $i$. 
We thus obtain $\underline x'_{g+1-i} = \max\{L_{first}(i,T)-l_r(i,T),0\} = \max\{-1,0\}=0$. Comparing this with formula 
(\ref{eq:phi'}), we see $l_c(i,T)=1$ and indeed $0= l_c(i,T)-1 + l_1 - l_1$.

On the other hand, if $i$ is in column $j \not = 1$, then $l_1 = L_{first}(i,T)$ and $l_j = l_r(i,T)$.
Since $j-1+l_1-l_j$ is certainly non-negative as it is the formula~\eqref{eq:phi'} for
	$\phi'\circ\ev$, we have $\underline x'_{g+1-i} = j - 1 + l_1 - l_j$ as desired in this case as well.    
    
    Hence, we have proven that the formulas for $\sigma\circ\phi$ and $\phi'\circ\ev$ agree.
\end{proof}

\section{Conjugation of Tableaux and Riemann-Roch Duality} \label{sec:res2}

We now prove a conjecture based on discussions with Sam Payne \cite{Payne2012}.

\begin{definition}
    We consider an involution $t\colon\RSYT\to\RSYT$, called
    \emph{conjugation} (also known as transposition), sending $T = (a_{ij})$ in $RSYT$ to $T^t=(a_{ji})$. 
\end{definition}

Unlike evacuation, the dimensions of a tableau $T$ are not fixed under conjugation.  In particular, an $m \times n$ standard Young tableau is 
sent to an $n \times m$ standard Young tableau.  Consequently, the corresponding divisors $\phi(T)$ and $\phi(T^t)$, have 
different ranks and degrees.  Nevertheless, there is a natural duality induced by this involution on rectangular standard Young tableaux.

\begin{theorem} \label{th:result2}
Suppose that $g = (g-d+r)(r+1)$ and that $T$ is a rectangular $(g-d+r) \times (r+1)$ standard Young tableau.  Let $\phi(T) = (r;\underline x_1,\ldots,\underline x_g)$, a rank $r$ and degree $d$ divisor $c$ in $Pic(\Gamma_g)$.

Then $\phi(T^t)$ is linear equivalent to the rank $g-d+r-1$ and degree $2g-2-d$ divisor $K - c$ in $Pic(\Gamma_g)$.  Here $K$ is the canonical divisor on $\Gamma_g$ that appears in the Tropical Riemann-Roch Theorem (Theorem \ref{TRR}).
\end{theorem}

To prove Theorem \ref{th:result2}, it suffices to prove the following two propositions.

\begin{proposition} \label{prop:trans1}
Assume the hypotheses of Theorem \ref{th:result2}, including the equality
$\phi(T) = (r;\underline x_1,\ldots,\underline x_g)$.
Define tuples $(\underline z_0,\underline z_1, \dots, \underline z_{g-1})$ and $(\underline y_1,\underline y_2,\dots, \underline y_g)$ as follows.
For $0 \leq i \leq g-1$, let $z_i = \#\{ j > i \mathrm{~in~the~last~row~or~last~column~of~}T\} + 1$.  Then, for $1 \leq i \leq g$, define $\underline y_i$ as 
\begin{eqnarray} \label{eq:yy} \underline y_i = z_{i-1} - \underline x_i - 2.\end{eqnarray}
Then we obtain $\phi(T^t) = (s; \underline y_1, \dots, \underline y_g)$ where $s = g-d+r-1$.
\end{proposition}

\begin{proposition} \label{prop:trans2}
Let $K$ be the canonical divisor on $\Gamma_g$, let $c$ be the $v_0$-reduced divisor in $Pic(\Gamma_g)$ represented by $(r;  \underline x_1, \dots, \underline x_g)$, and let 
$(s;  \underline y_1, \dots, \underline y_g)$ be as defined in Proposition \ref{prop:trans1}.
Then the $v_0$-reduced divisor equivalent to $K-c$ in $Pic(\Gamma_g)$ is represented by 
$(s;  \underline y_1, \dots, \underline y_g)$.
\end{proposition}

\begin{proof} [Proof of Proposition \ref{prop:trans1}]
Let $T^t$ denote the conjugate of $T$.  Following the logic of Proposition \ref{prop:x}, the first value of tuple $\phi(T^t)$ is one less than the number of columns of $T^t$.  Using the fact that $T^t \in SYT((g-d+r)^{(r+1)})$, we obtain that this value is $s = g-d+r-1$ as desired.  

It next suffices to show the equality $\underline x_i + \underline y_i = z_{i-1} - 2$ for all $1 \leq i \leq g$.  
Using formula (\ref{eq:phi}) for the $\underline x_i$'s and switching the roles of ``rows'' and ``columns'', or $T$ and $T^t$, to obtain a formula for the $\underline y_i$'s yields
\begin{eqnarray*}\underline x_i + \underline y_i &=& 
\left(r + l_r(i,T) - l_{c}(i,T) - L_{last}(i,T)\right) \\ &~& ~\hspace{0.1in} + \left(s + l_c(i,T) - l_{r}(i,T) - L_{last}(i,T^t)\right) \\ &=& r + s - L_{last}(i,T) - L_{last}(i,T^t), 
\end{eqnarray*}
where $L_{last}(i,T)$ (resp. $L_{last}(i,T^t)$) equals the number of cells in the last column of $T$ (resp. $T^t$) whose entries are strictly less than $i$.
It follows that
\begin{eqnarray*}
(s+1) - L_{last}(i,T) ~ &=& \#(\mathrm{cells}> (i-1)\mathrm{~in~the~last~column~of~T}), \mathrm{~and} \\
(r+1) - L_{last}(i,T^t) &=& \#(\mathrm{cells}> (i-1)\mathrm{~in~the~last~row~of~T}).
\end{eqnarray*}
Adding these two together, and subtracting $2$, we conclude that $$\underline x_i + \underline y_i = r + s - L_{last}(i,T) - L_{last}(i,T^t) = z_{i-1} - 2$$ as desired (keeping in mind that $z_{i-1}$ double-counts the unique cell in the last row and column). 
\end{proof}

Before proving Proposition \ref{prop:trans2}, we need one Lemma.

\begin{lemma} \label{Lem:XZ}
Let $T$ be an $(r+1)\times (s+1)$ rectangular standard Young tableau, and suppose that $\phi(T) = (r; \underline x_1, \underline x_2, \dots, \underline x_g)$.  Define $z_i$ as in Proposition \ref{prop:trans1}.  Then for $1 \leq i\leq g-1$, we have $\underline x_i = 0$ if the entry $i$ is in the last column of $T$ and $\underline x_i = z_i -1$ if the entry $i$ is in the last row of $T$.  If the entry $i$ is in neither the last row nor the last column, then $0 < \underline x_i < z_i - 2$.  Note in particular that $\underline x_i = z_i - 2$ is not possible.
\end{lemma}

\begin{proof}
The first statement was already shown in the proof of Proposition \ref{prop:x}.  To prove the second statement, assume that $i$ in the last row of $T$.  From formula (\ref{eq:phi}), we obtain $\underline x_i = r + (s+1) - l_c(i,T) - L_{last}(i,T)$ in this case.  Noting that $r+s+1$ equals the number of cells in the last row or last column of $T$, we see that $\underline x_i$ counts the number of cells in the last row or last column greater than $i$.  However, this is exactly the definition of $z_i - 1$ and hence the second statement is proven. 

In the event that the entry $i$ is in neither the last row nor the last column, then 
$\underline x_i = (r + s + 1) - l_c(i,T) - L_{last}(i,T)) - \alpha < z_i - 1 -\alpha$ where 
$\alpha = s+1 - l_r(i,T) \geq 1$.  Note that we have an inequality on the RHS instead of an equality this time because if $i$ is in column $l_c(i,T)$ then the bottom row of that column is greater than $i$ as opposed to merely equal to $i$.  We could additionally have entries in the bottom row to the left of column $l_c(i,T)$ that are greater than $i$, thus we have the desired inequality.
\end{proof}

\begin{proof} [Proof of Proposition \ref{prop:trans2}]
We first note that for graph $\Gamma_g$, the canonical divisor $K$ is simply 
given as $2v_1 + 2v_2 + \dots + 2v_{g-1}$.  Hence, the degree of $K-c$ is $d^t = 2(g-1) -d$ and the 
rank is $r^t = g - d + r - 1 = s$ by the Tropical Riemann-Roch Theorem.  In particular, we still have the equality $(g - d^t + r^t)(r^t + 1) = g$ since $(g - d^t + r^t)(r^t + 1) = (g - (2g - 2 - d) + (g - d + r - 1)) (g- d + r) = (r+1)(g-d+r)$. 

Secondly, we take $K-c$, where $c$ is represented by $(r; \underline x_1, \underline x_2,\dots, \underline x_g)$, and compute its $v_0$-reduction by successively firing the subgraphs 
$\gamma_g$, $\gamma_{g-1} \cup \gamma_g$, etc. enough times from right to left.
We define a sequence $(Z_{g-1},Z_{g-2},\dots, Z_{0})$ by letting $Z_{g-i}$ denote the number of chips on vertex $v_{g-i}$ after the subgraph $\gamma_{g-i+2} \cup \gamma_{g-i+3} \cup \dots \cup \gamma_g$ has been fired enough times.  In particular, $Z_{g-1}=2$ since in the divisor $K-c$, the vertex $v_{g-1}$ (before any firing) has $2$ chips on it.  We also obtain $Z_{g-2} = 3$ since $\gamma_g$ contains $2$ chips on $v_{g-1}$, and nowhere else, so after $\gamma_g$ is fired, there are $3$ chips on vertex $v_{g-2}$.  

We also define the sequence $(Y_g,Y_{g-1},\dots, Y_1)$ by letting $Y_{g-i}$ denote the counter-clockwise distance from $v_{g-i}$ of the unique chip on the loop $\gamma_{g-i}\setminus
\{v_{g-i-1},v_{g-i}\}$ after $\gamma_{g-i+1} \cup \gamma_{g-i+2}\cup \dots \cup \gamma_g$ has been fired, with the convention that $Y_{g-i} = 0$ if no such chip exists.
For example, the entry $g$ must be in the unique cell in the last row and last column, thus 
$\underline x_g = r + (s+1) - (r+1) - s = 0$ by formula (\ref{eq:phi}).  Consequently, the divisor $K-c$ has no chips on $\gamma_g \setminus \{v_{g-1},v_{g}\}$, and we obtain $Y_g = 0$ (no firings have yet taken place). 

We next observe that if $\underline x_{g-1}=1$, we fire $\gamma_g$ and obtain $Y_{g-1} = 0$ while if
$\underline x_{g-1} = 0$, we instead obtain $Y_{g-1} = 1$ after firing $\gamma_g$.  Note that formula (\ref{eq:phi}) yields
$\underline x_{g-1} = r + s - (r+1) - (s-1) = 0$ (resp. $\underline x_{g-1} = r + (s+1) - r - s = 1$) if $(g-1)$ is above (resp. to the left of) the cell containing $g$.  
Since entry $(g-1)$ must be in one of the two cells next to entry $g$, no other values for $\underline x_{g-1}$ are possible.

Having fired subgraphs $\gamma_{g-i+2} \cup \gamma_{g-i+3} \cup \dots \cup \gamma_g$ to 
clear chips from vertices $v_{g-i+1}, v_{g-i+2}, \dots v_g$ and ensure that there are no negative chips to the right of $v_{g-i}$, we next focus on the loop $\gamma_{g-i}$.  
Inductively, it contains $Z_{g-i}$ chips at $v_{g-i}$, $2$ chips at $v_{g-i-1}$ and $-1$ chips a counter-clockwise distance of $\underline x_{g-i}$ from 
$v_{g-i}$ (unless $\underline x_{g-i} = 0$ in which case the only chips are at $v_{g-i}$ and 
$v_{g-i-1}$) at this point.  Firing the subgraph $\gamma_{g-i+1} \cup \gamma_{g-i+2} \cup \dots \cup \gamma_g$, for $1\leq i \leq g-1$, we have three cases:

i) if $\underline x_{g-i} = 0$, then $Z_{g-i-1} = Z_{g-i} + 1$ and we have one chip left a counter-clockwise distance $Y_{g-i} = Z_{g-i} - 1$ from $v_{g-i}$.

ii) if $\underline x_{g-i} = Z_{g-i} - 1$, then we can move all of the chips to the left so we obtain 
$Z_{g-i-1} = Z_{g-i} + 1$ and $ Y_{g-i} = 0$.

iii) if $0 < \underline x_{g-i} < Z_{g-i} - 2$, then we fire the subgraph $\gamma_{g-i+1} \cup \gamma_{g-i+2} \cup \dots \cup \gamma_g$ in two steps, first eliminating the negative chip, and second moving all of the remaining chips off of vertex $v_{g-i}$.  We thus compute $Z_{g-i-1} = Z_{g-i}$ (no increase) and $ Y_{g-i} = Z_{g-i} - \underline x_{g-i} - 2$.

Following this process of $v_0$-reduction, it is clear that the tuple representing the $v_0$-reduction of $K-c$ is $(s; Y_1, Y_2, \dots Y_{g-1}, Y_g)$.  Furthermore, as we see from the above three cases, corresponding to firing $\gamma_{g-i+1} \cup \gamma_{g-i+2} \cup \dots \cup \gamma_g$, we have  
\begin{eqnarray} \label{eq:YZ} \underline x_{g-i} + Y_{g-i} = Z_{g-i-1} -2 \mathrm{~for~all~}1 \leq i \leq g-1.\end{eqnarray}
We therefore finish the proof by proving simultaneously $Z_{g-i-1} = z_{g-i-1}$ and $Y_{g-i} = \underline y_{g-i}$, for all $0 \leq i \leq g-1$, by double-induction.

Note that the base cases $z_{g-1} = Z_{g-1} = 2$, $z_{g-2} = Z_{g-2} = 3$, $\underline y_{g} = Y_g = 0$ and $\underline y_{g-1} = Y_{g-1} = 1- \underline x_{g-1}$ were shown above.
Assume now by induction that $z_{g-i} = Z_{g-i}$ and $\underline y_{g-i+1} = Y_{g-i+1}$.  If the entry $i$ is in the last column or last row of $T$, then $z_{g-i-1} = z_{g-i}+1$ by definition, and 
$\underline x_{g-i} = 0$ or $\underline x_{g-i} = z_{g-i}-1 = Z_{g-i}-1$ by Lemma \ref{Lem:XZ} and the induction hypothesis.  It follows by cases (i) and (ii) above and the induction hypothesis again that $Z_{g-i-1} = Z_{g-i} + 1 = z_{g-i} + 1 = z_{g-i-1}$ in these cases.  

On the other hand, if the entry $i$ is not in the last row or last column of $T$, then 
$z_{g-i-1} = z_{g-i}$ by definition, and $0 < \underline x_{g-i} < Z_{g-i} - 2$ by Lemma \ref{Lem:XZ}.  Then by case (iii), we have $Z_{g-i-1} = Z_{g-i}$ and the equality $z_{g-i-1} = Z_{g-i-1}$ follows again by induction. 

Finally, in both cases, the identities 
 (\ref{eq:yy}) and (\ref{eq:YZ}) imply the equality of $\underline y_{g-i}$ and $Y_{g-i}$, thus  finishing this inductive proof.   
\end{proof}

\begin{example}
Let $c=\phi(T)$ denote the chip-configuration 
$(2; \underline 2, \underline 3, \underline 1, \underline 0, \underline 1, \underline 0)$, i.e. the 
running example, Example \ref{ex:running}, where $T$ is the tableau     
   \[
    T = 
    \begin{tabular}{|r|r|r|}
        \hline
		1 & 3 & 4 \\
		\hline
		2 & 5 & 6 \\
		\hline
	\end{tabular}~.
    \]
By successively firing (i) $\gamma_6$, (ii) $\gamma_5 \cup \gamma_6$, 
(iii) $\gamma_4 \cup \gamma_5 \cup \gamma_6$, 
(iv) $\gamma_4 \cup \gamma_5 \cup \gamma_6$ again, 
(v) $\gamma_3 \cup \dots \cup \gamma_6$, 
(vi) $\gamma_2 \cup \dots \cup \gamma_6$,
and finally (vii) $\gamma_2 \cup \dots \cup \gamma_6$ again, we 
$v_0$-reduce the divisor $K-c$ into the divisor
$(1; \underline 1, \underline 0, \underline 1, \underline 2, \underline 0, \underline 0)$.  
See Figure \ref{fig:conj-chips}.  Using formula (\ref{eq:phi}) to compute $\phi(T^t)$, where 
    \[
    T^t = 
    \begin{tabular}{|r|r|}
        \hline
		1 & 2 \\
		\hline
		3 & 5 \\
		\hline
		4 & 6 \\
		\hline
	\end{tabular}~,
    \]
we do indeed obtain that $\phi(T^t)$ is the $v_0$-reduced divisor equivalent to $K-\phi(T)$.

\begin{figure}[h!tbp]
\centering
\subfigure{\includegraphics{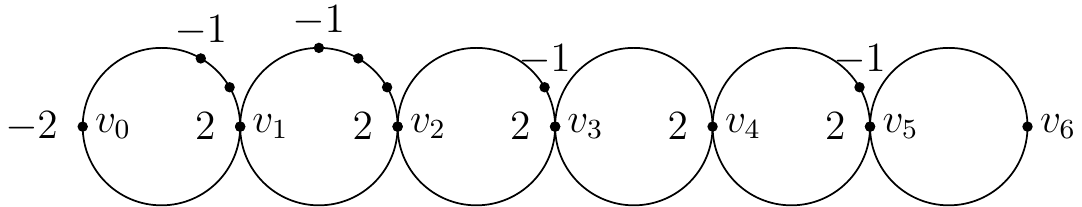}}
(i) \subfigure{\includegraphics{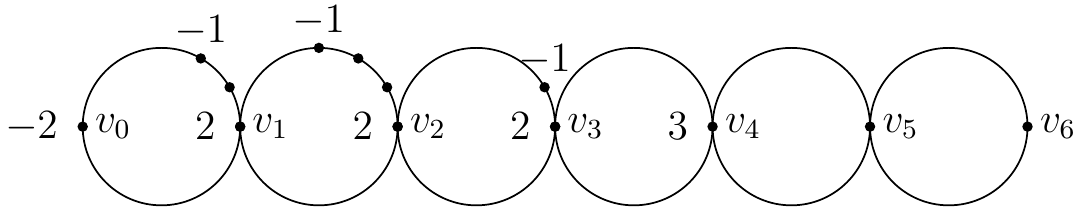}}
(ii) \subfigure{\includegraphics{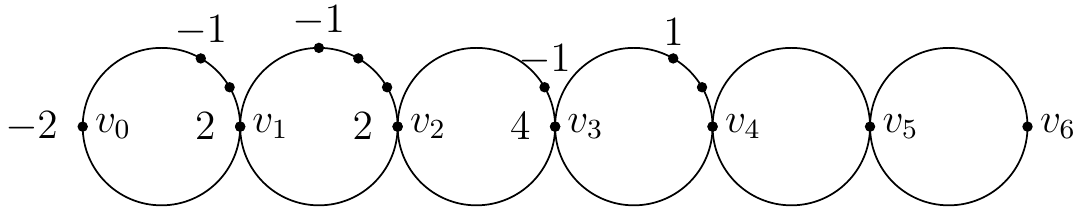}}
(iii) \subfigure{\includegraphics{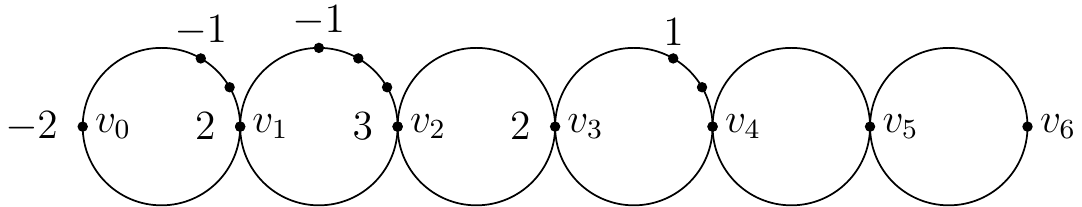}}
(iv) \subfigure{\includegraphics{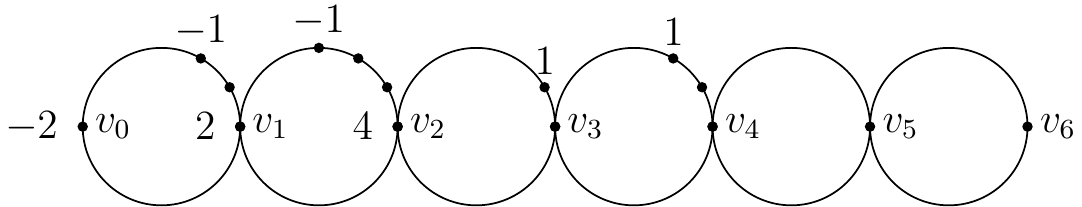}}
(v) \subfigure{\includegraphics{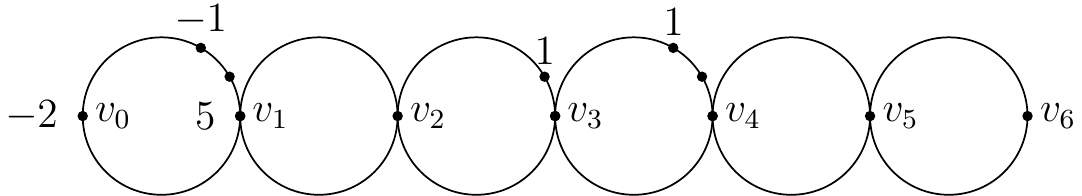}}
(vi) \subfigure{\includegraphics{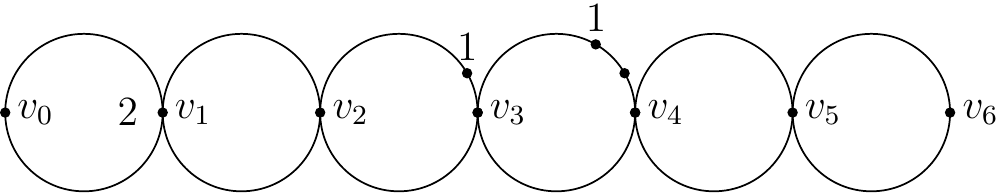}}
(vii) \subfigure{\includegraphics{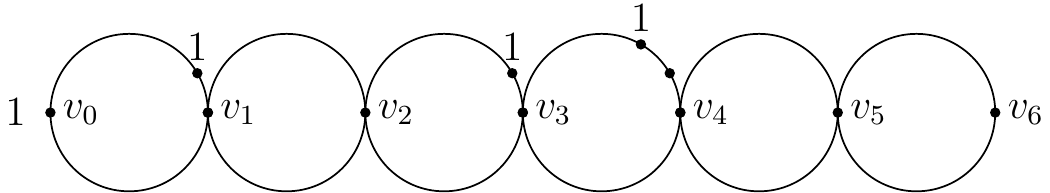}}
\caption{Successively $v_0$-reducing the divisor $(K-c)$ to $\phi(T^t)$}
\label{fig:conj-chips}
\end{figure}

\end{example}

\section{Acknowledgments}
This research was conducted at the $2011$ and $2012$ summer REU (Research Experience for
Undergraduates) programs at the University of Minnesota, Twin Cities, and was
partially funded by NSF grants DMS-$1001933$, DMS-$1067183$, and DMS-$1148634$. 
This led to an REU report \cite{Agrawal2011} predating this article.
The authors would
like to thank Profs.\ Vic Reiner and Pavlo Pylyavskyy, who along with author Gregg Musiker
directed the program, for their support. We would also like to thank the anonymous referees for their comments and suggestions.
\bibliographystyle{amsalpha}
\bibliography{bibliography}

\end{document}